\documentclass[11pt,letterpaper]{amsart}
\usepackage{eucal}
\usepackage{graphicx}
\usepackage{bbold}
\usepackage{amsopn,amssymb}

\newcommand{\Z}{\mathbb{Z}}
\newcommand{\R}{\mathbb{R}}
\newcommand{\C}{\mathbb{C}}

\newcommand{\Q}{\mathbb{Q}}
\newcommand{\X}{\mathcal{X}}
\newcommand{\B}{\mathcal{B}}
\newcommand{\W}{\mathcal{W}}
\newcommand{\sslash}{/\!\!/}

\newcommand{\T}{\mathcal{T}}

\newcommand{\QF}{\mathcal{QF}}

\newcommand{\CP}{\mathbb{CP}}
\renewcommand{\H}{\mathbb{H}}
\newcommand{\ML}{\mathcal{ML}}
\newcommand{\PML}{\mathbb{P}\mathcal{ML}}
\renewcommand{\Tilde}[1]{\widetilde{#1}}
\newcommand{\scc}{\mathcal{S}}

\newcommand{\boldpoint}[1]{\smallskip\par\noindent\textbf{#1}}

\newcommand{\param}{{\mathchoice{\mkern1mu\mbox{\raise2.2pt\hbox{$\centerdot$}}\mkern1mu}{\mkern1mu\mbox{\raise2.2pt\hbox{$\centerdot$}}\mkern1mu}{\mkern1.5mu\centerdot\mkern1.5mu}{\mkern1.5mu\centerdot\mkern1.5mu}}}

\DeclareMathOperator{\Hom}{Hom}
\DeclareMathOperator{\PSL}{\mathrm{PSL}}
\DeclareMathOperator{\SL}{\mathrm{SL}}
\DeclareMathOperator{\gr}{gr}

\DeclareMathOperator{\Log}{Log}
\DeclareMathOperator{\tr}{tr}

\DeclareMathOperator{\cone}{Cone}
\DeclareMathOperator{\isom}{Isom}
\DeclareMathOperator{\hol}{hol}
\DeclareMathOperator{\arccosh}{arccosh}

\numberwithin{equation}{section}

\theoremstyle{plain}
\newtheorem{thm}{Theorem}[section]
\newtheorem{cor}[thm]{Corollary}

\newtheorem{lem}[thm]{Lemma}

\newtheorem{bigthm}{Theorem}

\theoremstyle{definition}

\theoremstyle{remark}
\newtheorem*{remark}{Remark}

\begin{document}

\title{Bers slices are Zariski dense}
\author{David Dumas and Richard P. Kent IV}
\date{May 5, 2009}
\thanks{Both authors supported by NSF postdoctoral research fellowships.}
\maketitle

Let $S$ be a closed oriented surface of genus $g \geq 2$, and let
$\X_\C(S)$ denote the variety of $\SL_2\C$-characters of $\pi_1(S)$.
The quasi-Fuchsian space $\QF(S)$ is an open subset of the smooth
locus of $\X_\C(S)$ equipped with a biholomorphic parametrization by
a product of Teichm\"uller spaces
$$ Q : \T(S) \times \T(\overline{S}) \to \QF(S). $$
The space $\QF(S)$ parametrizes the quasi-Fuchsian representations of
$\pi_1(S)$, and the parametrization $Q$ is due to Bers \cite{bers}.

Picking a point in one factor gives a \emph{Bers slice}
$$ \B_Y = Q\left( \T(S) \times \{ Y \} \right) \subset \QF(S).$$ Each
Bers slice is a holomorphically embedded copy of Teichm\"uller space
within $\X_\C(S)$.  While it follows that $\B_Y$ can be locally
described as the common zero locus of finitely many analytic functions
on $\X_\C(S)$, it is known that the Bers slice is \emph{not} a locally
algebraic set \cite{dumas-kent}---this is used to show that W.~Thurston's skinning map
is not a constant function \cite{dumas-kent}.

We prove a stronger result about the transcendence of $\B_Y$:

\begin{bigthm}
\label{thm:main}
Let $S$ be a closed surface of genus $g \geq 2$.  For any $Y \in
\T(S)$, the Bers slice $\B_Y$ is Zariski dense in $\X_{\C}(S)$.
\end{bigthm}

The proof is roughly as follows: The Zariski closure of the Bers slice
$\B_{Y}$ contains an analytic subvariety $\W_Y \subset \X_\C(S)$ of
complex dimension $-\frac{3}{2}\chi(S)$ which consists of holonomy
representations of projective structures on $Y$.  By analyzing its
real points, we show that $\W_Y$ accumulates on a
$(-3\chi(S)-1)$--dimensional subset of the logarithmic limit set of
$\X_\C(S)$, whereas a proper subvariety of $\X_\C(S)$ accumulates on a
set of strictly smaller dimension.  Thus $\W_Y$, and hence $\B_Y$, is
Zariski dense.

\boldpoint{Acknowledgments.} The authors thank Ian Agol, Alexander
Goncharov, and Peter Shalen for helpful conversations related to this work.

\section{Logarithmic Limit Sets}

For an affine algebraic variety $V \subset \C^n$, let $\Log |V|
\subset \R^n$ denote the \emph{amoeba} of $V$: the set of points $\{
(\log |z_1|, \ldots, \log |z_n|) \: | \: z \in V, z_i \neq 0\}$.

The \emph{logarithmic limit set} of $V$, denoted $V^\infty$, is the
set of limit points of $\Log |V|$ on the positive projectivization
$S^{n-1} = (\R^n \setminus \{0\}) / \R^+$.  Equivalently, considering
$S^{n-1}$ as the boundary of the unit ball $B^n$, we have
\begin{equation*}
\label{eqn:closure}
 V^\infty = \overline{f(\Log|V|)} \cap S^{n-1},
\end{equation*}
where $f : \R^n \to B^n$ is the rescaling map $ f(x) = \frac{x}{1 +
  \|x\|}$.

From this description, we see that $V^\infty$ is the boundary of a
  \emph{logarithmic compactification} of $V^* = V \cap (\C^*)^n$,
  which we denote by $\overline{V} = V^* \sqcup V^\infty$.
  Explicitly, a sequence $z^{(i)} \in V^*$ converges to the ray $[p] =
  \R^+\cdot p \in S^{n-1}$, where $p \in \R^n - \{0\}$, if there exists a
  sequence $c_i \in \R^+$ such that $c_i \to 0$ and 
\begin{equation}
\label{eqn:logarithmic}
 \lim_{i \to \infty} \left(c_i \log |z^{(i)}_1|, \ldots, c_i \log |z^{(i)}_n|
  \right) = (p_1, \ldots, p_n).
\end{equation}

Given a subset $E \subset S^{n-1}$, let $\cone(E) = (\R^+\cdot E) \cup
\{0\} \subset \R^n$ denote the \emph{cone} on $E$.  Note that $E =
\cone(E) \cap S^{n-1}$.  Some properties of the logarithmic limit set
$V^\infty$ are most easily expressed in terms of $\cone(V^\infty)$.
For example, G.~Bergman proved the fundamental structure theorem:
\begin{thm}[\cite{bergman}]
\label{thm:log-dimension}
  If $V \subset \C^n$ is an affine algebraic variety of complex
  dimension $d$, then $\cone(V^\infty)$ is contained in a finite union
  of $d$-dimensional subspaces of $\R^n$ defined over $\Q$.
\end{thm}

Equivalently, the logarithmic limit set $V^\infty$ is contained in a
finite union of \emph{rational great $(d-1)$-spheres}.

Strengthening the theorem above, Bieri and Groves \cite{bieri-groves}
showed that $\cone(V^\infty)$ is a finite union of rational polyhedral
convex cones---sets of the form $\{ \sum_{i=1}^k c_i x_i \: | \: c_i
\geq 0\}$ where $x_i \in \Q^n$---and that at least one of these cones
spans a subspace of dimension $d$.  However, for
our purposes, the dimension estimate of Theorem
\ref{thm:log-dimension} will suffice.

We extend the definition of logarithmic limit sets to apply to subsets
of algebraic varieties: for $T \subset V$, we denote by $T^\infty$ the
boundary of $T \cap V^*$ in $\overline{V}$.  In our cases of interest,
$T$ will be a noncompact and properly embedded submanifold of the
smooth points of $V^*$, and in particular $T^\infty$ will be nonempty.

\section{Character varieties}

Let $\X_\C(S) = \Hom(\pi_1(S), \SL_2\C) \sslash \SL_2\C$ denote the
$\SL_2\C$ character variety of the compact surface $S$, which is an
irreducible affine variety of dimension $6g-6$ (see
\cite{culler-shalen} \cite{goldman:topological-components}).  For a suitable finite
subset $\{ \gamma_1, \ldots, \gamma_N\}$ of $\pi_1(S)$, the variety
$\X_\C(S)$ can be realized as the image of $\Hom(\pi_1(S),\SL_2\C)$ by
the trace map
$$ \rho \mapsto \left( \tr \rho(\gamma_1), \ldots, \tr \rho(\gamma_N)
\right ).$$ We call such $\{\gamma_i\}$ an \emph{embedding set for
  $\X_\C(S)$}.  For example, given a generating set for $\pi_1(S)$
with $n$ elements, the set of words of length at most $n$ in these
generators is an embedding set for $\X_\C(S)$ \cite[\S1.4]{culler-shalen}.
The traces of the images of an embedding set
determine the traces of all elements of $\rho(\pi_1(S))$---that is,
they determine the \emph{character} of $\rho$.

We will only consider the characters of nonelementary representations
$\rho : \pi_1(S) \to \SL_2\C$, which are in one-to-one correspondence
with conjugacy classes of such representations.  When convenient, we
blur the distinction between a character and representatives of the
associated conjugacy class of representations.

Each marked hyperbolic structure on $S$ has a corresponding Fuchsian
representation $\rho : \pi_1(S) \to \SL_2\R$ (one of the finitely many
lifts of $\rho_0 : \pi_1(S) \to \PSL_2\R \simeq \isom^+(\H^2)$).  In
this way, the Teichm\"uller space $\T(S)$ can be identified with a
connected component of the set of real points $\X_\R(S) \subset
\X_\C(S)$ \cite{goldman:topological-components}.  Similarly, the
quasi-Fuchsian space $\QF(S)$ of equivariant quasiconformal
deformations of Fuchsian representations is identified with an open
neighborhood of $\T(S)$ in $\X_\C(S)$ \cite{marden} \cite{sullivan}.

\section{Measured laminations and Thurston's compactification}
Let $\ML(S)$ denote the space of measured geodesic laminations on $S$,
which is a piecewise--linear manifold homeomorphic to $\R^{6g-6}$.
Let $\scc$ denote the set of isotopy classes of homotopically
nontrivial simple closed curves on $S$, and note that each element of
$\scc$ corresponds to a conjugacy class $[\gamma] \subset \pi_1(S)$.
Then $\R^+ \times \scc$ is naturally a dense subset of $\ML(S)$
consisting of weighted simple closed geodesics.

For a suitable finite subset $\{ \gamma_1, \ldots, \gamma_N \}$ of
$\pi_1(S)$, we have a piecewise--linear embedding $\ML(S) \to \R^{N}$,
$$ \lambda \mapsto \left ( i(\lambda, \gamma_1), \ldots, i(\lambda,
  \gamma_N) \right )$$ where $i(\lambda, \gamma)$ is the
\emph{intersection number}, the minimum mass assigned to a
representative of the isotopy class of $\gamma$ by the transverse
measure of $\lambda \in \ML(S)$.   (The intersection number for
simple curves is
discussed in \cite[Ch.~5]{flp}; for closed but possibly
self-intersecting curves, see \cite{bonahon:currents}.)  We call such
$\{\gamma_1, \ldots \gamma_N\} \subset \pi_1(S)$ an \emph{embedding
  set for $\ML(S)$}.  An example of an embedding set with $9g-9$
elements is described in \cite[\S6.4]{flp}.  The image of
such an embedding is a piecewise--linear cone, and the space $\PML(S)$
of rays in $\ML(S)$ can be identified with $\ML(S) \cap S^{N-1}$.

Thurston's compactification of Teichm\"uller space adjoins $\PML(S)$
as the boundary of $\T(S)$ according to the asymptotic behavior of
hyperbolic lengths of geodesics.  Specifically, a sequence $X_n \in
\T(S)$ converges to $[\lambda] = \R^+\lambda \in \PML(S)$ if and only
if there is a sequence $c_n \in \R^+$ such that $c_n \to 0$ and
\begin{equation}
\label{eqn:thurston}
 \lim_{n \to \infty} c_n \ell(\alpha,X_n) \to i(\lambda,\alpha)
\end{equation}
for all closed curves $\alpha$.  Thurston showed that the same
compactification is obtained if \eqref{eqn:thurston} is required only
for the finitely many $\alpha$ in an embedding family for
$\ML(S)$.  For details, see \cite[Ch.~8]{flp} \cite[Thm.~18]{bonahon:currents}.

\section{Logarithmic limit sets and Thurston's compactification}

We fix throughout a finite set $\{ \gamma_1, \ldots, \gamma_N \}
\subset \pi_1(S)$ that is both an embedding set for $\X_\C(S)$ and for
$\ML(S)$ (e.g.~the union of the example embedding sets described
above).  We use this family to identify $\X_\C(S)$ and $\ML(S)$ with
their embedded images in $\C^N$ and $\R^N$, respectively.  Similarly,
we regard $\PML(S)$ as a subset of $S^{N-1}$.

With this embedding of $\X_\C(S)$, we can consider the logarithmic
limit sets $\X_\C(S)^\infty$ and $\T(S)^\infty$ in $S^{N-1}$.

\begin{lem}\label{lem:thurston-as-log}
We have $\T(S)^\infty = \PML(S)$, and the logarithmic compactification
of $\T(S)$ is identical to the Thurston compactification.
\end{lem}

\begin{proof}
If $\rho_X \in \Hom(\pi_1(S),\SL_2\R)$ represents $X \in \T(S)$, then its
character is real and satisfies $|\tr \rho(\gamma)| \geq 2$ for all
$\gamma$, with equality only when $\gamma = 1$.  The hyperbolic length and
trace of $\alpha \in \scc$ are related by
$$ \ell(\alpha, X) = 2 \arccosh \frac{|\tr \rho(\alpha)|}{2}.$$ 
Since $\arccosh \frac{|t|}{2} = \log |t| + o(1)$ as
$t \to \infty$, then for any $X_i \in \T(S)$ and $c_i \to 0$,
the sequences $c_i \arccosh \frac{| \tr \rho_{X_i}(\gamma) |}{2}$ and
$c_i \ell(\gamma,X_i)$ have the same limit, and one converges if and
only if the other does.  Applying this to the $\gamma_1, \ldots, \gamma_N$
demonstrates the equivalence of convergence in the Thurston
compactification \eqref{eqn:thurston} and convergence in the
logarithmic one \eqref{eqn:logarithmic}.
\end{proof}

\begin{remark}
  The lemma above is essentially the same as Theorem III.3.2 of
  \cite{morgan-shalen}, which identifies the Morgan--Shalen and
  Thurston compactifications of Teichm\"uller space.  The definition
  of the Morgan--Shalen compactification is similar to that of the
  logarithmic compactification, except that it is defined using the
  function $\log (2 + |\tr \rho(\param) |)$ rather than $\log | \tr
  \rho(\param)|$; and by considering the limiting behavior of the
  traces of \emph{all} elements of $\pi_1(S)$, rather than a finite
  subset.  The difference in the logarithmic scaling function does not
  affect projective limits, and since finitely many intersection
  number functions embed $\ML(S)$ into $\R^n$, Lemma
  \ref{lem:thurston-as-log} follows.  Morgan and Shalen note that
  their Theorem III.3.2 is, in turn, equivalent to results of
  Thurston.  The connection between Thurston's compactification and
  logarithmic limit sets is explained in detail in recent papers of
  Alessandrini (\cite{alessandrini} and \cite[\S6]{alessandrini2}), to
  which we refer the reader for a complete discussion.
\end{remark}

\section{Quadratic differentials and holonomy}
\label{sec:qdhol}

For any $Y \in \T(S)$, let $Q(Y) \simeq \C^{3g-3}$ denote the vector
space of holomorphic quadratic differentials on the Riemann surface
$Y$.  We identify this vector space with the set of \emph{complex
  projective structures} on $Y$: conformal atlases whose transition
functions are M\"obius maps.  Under this identification, the zero
quadratic differential corresponds to the Fuchsian uniformization
$\Tilde{Y} \simeq \H$, and, more generally, $\phi \in Q(Y)$
corresponds to a projective structure whose developing map $f_\phi :
\H \to \CP^1$ has Schwarzian derivative $S(f_\phi) = \phi$.

The \emph{holonomy map} is the holomorphic map
$$ \hol: Q(Y) \to \X_\C(S)$$ which sends a quadratic differential
$\phi$ to the conjugacy class of the holonomy representation of the
associated projective structure on $Y$.  Here we implicitly lift these
holonomy representations from $\PSL_2\C$ to $\SL_2\C$, which requires
the choice of a spin structure on $Y$, or equivalently, a choice of
cohomology class in $H^1(Y, \pi_1(\PSL_2\C))$.  The particular choice among
the finite set of spin structures will not concern us.

The map $\hol$ is a proper holomorphic embedding, so its image
$$ \W_Y = \hol(Q(Y)) \subset \X_\C(S)$$ is an analytic subvariety
\cite[\S11.4]{gkm} (also see \cite[\S5.7]{dumas:survey} for the
history of this theorem).  We also have $\B_Y \subset \W_Y$, and in fact, an
open set $B \subset Q(Y)$ (the \emph{Bers embedding of $\T(Y)$}) maps
biholomorphically to $\B_Y$ under $\hol$.  This leads to a key
observation used in the proof of Theorem \ref{thm:main}:
\begin{lem}[{\cite{dumas-kent}}]
Any analytic subvariety of $\X_\C(S)$ containing $\B_Y$ also contains
$\W_Y$.  In particular, the Zariski closure of $\B_Y$ contains $\W_Y$.
\end{lem}
Thus to prove the main theorem, it suffices to show that $\W_Y$ is
Zariski dense in $\X_\C(S)$.

\section{Grafting}

We now describe certain points on $\W_Y$ in terms of grafting.

Equip $X \in \T(S)$ with its hyperbolic metric, and let $\gamma$ be a
simple closed geodesic on $X$.  Removing $\gamma$ from $X$ and
replacing it with a Euclidean cylinder $\gamma \times [0,t]$, we
obtain a new surface with a well-defined conformal structure.  This is
the \emph{grafting of $X$ by $t \gamma$}, denoted $\gr_{t \gamma}X$.
By adjoining multiple cylinders, grafting extends naturally to
measured laminations $\lambda = \sum_i c_i \gamma_i$ supported on
unions of disjoint simple closed geodesics.  The case when $c_i \in 2 \pi
\Z$ will be of particular interest to us, and we let $2 \pi \ML_\Z(S)$
be the set of all such $2\pi$-integral measured laminations.

\begin{thm}[{Tanigawa, \cite{tanigawa}}]
For each $\lambda \in 2 \pi \ML_\Z(S)$, the map $\gr_\lambda : \T(S)
\to \T(S)$ is a diffeomorphism.
\end{thm}

\begin{remark}
Scannell and Wolf have shown that the same result holds for all
$\lambda \in \ML(S)$ \cite{scannell-wolf}.
\end{remark}

So we have an inverse map $\gr_\lambda^{-1} : \T(S) \to \T(S)$.
Goldman showed that grafting can be used to describe the intersection
$\W_Y \cap \T(S)$ explicitly:

\begin{thm}[{Goldman, \cite{goldman}}]
\label{thm:goldman}
For each $Y$ we have $$\W_Y \cap \T(S) = \{ \gr_\lambda^{-1}(Y) \: |
\: \lambda \in 2 \pi \ML_\Z(S) \}.$$
\end{thm}

The map $\lambda \mapsto \gr_\lambda^{-1}(Y)$ is injective (see
\cite{dumas-wolf}), establishing a bijection between $\W_Y \cap \T(S)$
and $\ML_\Z(S)$.  In particular, the set $\W_Y \cap \T(S)$ is
infinite, which is used in \cite{dumas-kent} to show that $\W_Y$ is
not an algebraic variety.

\section{Antipodal limits}

Each quadratic differential $\phi \in Q(Y)$ defines a pair of
orthogonal singular foliations of the Riemann surface $Y$, the
\emph{horizontal} and \emph{vertical} foliations, whose leaves
integrate the distributions $\{ v \in TY \:|\: \phi(v) \geq 0 \}$ and $\{
v \in TY \:|\: \phi(v) \leq 0\}$, respectively.  Straightening the leaves of the
horizontal foliation with respect to the hyperbolic metric yields a
measured lamination $\Lambda_Y(\phi) \in \ML(S)$ (see \cite{levitt}),
and Hubbard and Masur showed that $\Lambda_Y : Q(Y) \to \ML(S)$ is a
homeomorphism \cite{hubbard-masur}.  Similarly, the map $\phi \mapsto
\Lambda_Y(-\phi)$ corresponds to straightening the vertical foliation,
and is also a homeomorphism.

Define the \emph{$Y$-antipodal map} $i_Y : \ML(S) \to \ML(S)$ by
$$ i_Y(\lambda) = \Lambda_Y( - \Lambda_Y^{-1}(\lambda)).$$
Thus $i_Y$ is an involutive homeomorphism that exchanges the
laminations corresponding to vertical and horizontal foliations of
quadratic differentials on $Y$.  It is easy to see that for all $c>0$,
we have $i_Y(c \lambda) = c i_Y(\lambda)$, and thus $i_Y$ descends to
a homeomorphism $\PML(S) \to \PML(S)$, which we also call $i_Y$.

We need the following result from \cite{dumas:antipodal}.

\begin{thm}
\label{thm:antipodal}
Let $\lambda_n \to \infty$ be a divergent sequence in $\ML(S)$ such
that the rays $[\lambda_n] \in \PML(S)$ converge to $[\lambda]$.  Then
$$\lim_{n \to \infty} \gr_{\lambda_n}^{-1}(Y) = [ i_Y(\lambda) ]$$
in the Thurston compactification.
\end{thm}

\begin{remark}
  While we have stated the above Theorem for arbitrary sequences of
  laminations, for our purposes it is enough to consider limits of
  sequences $\gr_{2 \pi n \gamma}^{-1}(Y)$, where $\gamma \in \scc$
  and $n \to \infty$.  The proof of Theorem \ref{thm:antipodal} in
  this special case is somewhat simpler.
\end{remark}

Combining this with the grafting description of $\W_Y \cap \T(S)$, we have:

\begin{cor}
\label{cor:accumulation}
For any $Y \in \T(S)$, the set $\W_Y \cap \T(S)$ accumulates on the
entire Thurston boundary $\PML(S) = \T(S)^\infty$.
\end{cor}

\begin{proof}
By Theorems \ref{thm:antipodal} and \ref{thm:goldman}, the closure of
$\W_Y \cap \T(S)$ contains all points $[\mu] \in \PML(S)$ of the form
$i_Y([\lambda])$, where $\lambda \in 2 \pi \ML_\Z(S)$.  Since the rays
$\{ [\lambda] \: | \: \lambda \in 2 \pi \ML_\Z(S)\}$ are dense in
$\PML(S)$, and $i_Y : \PML(S) \to \PML(S)$ is a homeomorphism, we find
that $\overline{\W_Y \cap \T(S)}$ is also dense in $\PML(S)$.  Since
$\overline{\W_Y \cap \T(S)}$ is closed, the Corollary follows.
\end{proof}

\section{Density}

We have seen in \S\ref{sec:qdhol} that Theorem \ref{thm:main} follows from

\begin{thm}
The Zariski closure of $\W_Y$ is $\X_\C(S)$.
\end{thm}

\begin{proof}
Let $V \subset \X_\C(S)$ denote the Zariski closure of $\W_Y$.  By
Corollary \ref{cor:accumulation} and Lemma \ref{lem:thurston-as-log},
we have $\PML(S) \subset V^\infty$ and thus $\ML(S) \subset
\cone(V^\infty)$.  In particular $\cone(V^\infty)$ contains an open
subset of a subspace of $\R^N$ of dimension $-3 \chi(S)$.

On the other hand, by Theorem \ref{thm:log-dimension}, we have that
$\cone(V^\infty)$ is contained in a finite union of subspaces of real
dimension $d = \dim_\C(V)$.  Thus $\dim_\C V \geq -3 \chi(S)
= \dim_\C \X_\C(S)$.  Since the variety $\X_\C(S)$ is irreducible (see
\cite{goldman:topological-components}), the Theorem follows.
\end{proof}

The dimension count above also shows:

\begin{cor}
  For any $Y \in \T(S)$, the set $\{ \gr_{\lambda}^{-1}(Y) \: | \:
  \lambda \in 2 \pi \ML_\Z(S) \}$ is Zariski dense in $\X_\R(S)$.  The
  same is true of $\{ \gr_{\lambda}^{-1}(Y) \: | \: \lambda \in E\}$
  for any $E \subset 2 \pi \ML_\Z(S)$ that accumulates on an open
  subset of $\PML(S)$. 
\end{cor}

\nocite{}

\bigskip

\noindent Department of Mathematics, Statistics, and Computer Science,
University of Illinois at Chicago
\newline \noindent \texttt{ddumas@math.uic.edu}

\medskip

\noindent Department of Mathematics, Brown University
\newline \noindent  \texttt{rkent@math.brown.edu} 

\end{document}